\documentclass[12pt]{amsart}
\usepackage{mathpazo,fullpage}%,backref}
\usepackage{latexsym}
\usepackage{amssymb}
\usepackage{url}
\usepackage{pictexwd,dcpic}
\usepackage{graphicx}
\usepackage{pinlabel}

\newtheorem{thm}{Theorem}
\newtheorem{prop}[thm]{Proposition}
\newtheorem{cor}[thm]{Corollary}
\newtheorem{lem}[thm]{Lemma}

\theoremstyle{remark}
\newtheorem{rem}[thm]{Remark}
\newtheorem{ex}[thm]{Example}

\theoremstyle{definition}

\newcommand{\Q}{\mathbb{Q}}
\newcommand{\R}{\mathbb{R}}
\newcommand{\C}{\mathbb{C}}
\newcommand{\Z}{\mathbb{Z}}

\title{Geometric formality and non-negative scalar curvature}
\author{D.~Kotschick}
\address{Mathematisches Institut, {\smaller LMU} M\"unchen, Theresienstr.~39, 80333~M\"unchen, Germany}
\email{dieter@member.ams.org}
\date{December 21, 2012; \copyright{\ D.~Kotschick 2012}}
\thanks{Research done at the Institute for Advanced Study in Princeton with the support of The Fund For Math and The Oswald Veblen Fund.}

\begin{document}

\begin{abstract}
We classify manifolds of small dimensions that admit both, a Riemannian metric of non-negative scalar curvature, 
and a -- {\it a priori} different -- metric for which all wedge products of harmonic forms are harmonic. 
For manifolds whose first Betti numbers are sufficiently large, this classification extends to higher dimensions.
\end{abstract}

\maketitle

\bigskip
\bigskip

\section{Introduction}\label{s:introduction}

A closed orientable manifold is called geometrically formal if it admits a Riemannian metric for which all wedge products of 
harmonic forms are harmonic. Such a metric is called a formal metric. This concept was introduced in~\cite{Formal} as a Riemannian
analogue, and sharpening, of formality in rational homotopy theory. It has turned out to be a very strong property, so much so that, 
under suitable assumptions, one can hope to classify geometrically formal manifolds, and perhaps even all the formal metrics on them.

Harmonic forms with respect to formal metrics have constant lengths and inner products.
In conjunction with curvature conditions, this implies classification results. For example, it follows from a 
result of Seaman~\cite{Sea} that a four-manifold with a formal metric of positive sectional curvature must have
definite intersection form. This rules out $S^2\times S^2$, showing that the Hopf conjecture holds for formal
metrics. This conclusion was recently rediscovered by B\"ar~\cite{B}, who set out to classify formal metrics of non-negative 
sectional curvature on four-manifolds up to isometry.

In this paper we take a different approach, and look at geometrically formal manifolds admitting some, possibly non-formal,
Riemannian metric of non-negative scalar curvature. In small dimensions we classify such manifolds topologically.
This is possible because the existence of a metric of non-negative scalar curvature is a strong condition in small dimensions 
due to the Gauss--Bonnet theorem (in dimension $2$), the Thurston--Hamilton--Perelman geometrization (in dimension $3$),
and Seiberg--Witten theory (in dimension $4$).

Among closed orientable surfaces, only the sphere and torus are geometrically formal. On the sphere all metrics are formal,
whereas on the torus a metric is formal if and only if it is flat; cf.~\cite{Formal}. Geometrically formal three-manifolds were 
classified in~\cite[Section~4]{Formal}. If one restricts to manifolds admitting a metric of non-negative scalar curvature, 
then there are very few examples, as we will show in Section~\ref{s:3}:

%\pagebreak

\begin{thm}\label{t:main3}
Any geometrically formal closed oriented three-manifold admitting a Riemannian metric of non-negative scalar curvature is 
diffeomorphic to:
\begin{enumerate}
\item a connected sum of spherical space forms,
\item $S^1\times S^2$, or
\item a flat manifold.
\end{enumerate}
Conversely, all these manifolds have Riemannian metrics which are simultaneously formal and of non-negative scalar curvature.
\end{thm}
Let us emphasize once more that we do not assume at the outset that there is a formal metric of non-negative scalar curvature - 
the formal metric and the metric with non-negative scalar curvature are allowed to be distinct. But the conclusion is that there is 
indeed a metric with both properties.

We can almost generalize this result to four dimensions:
\begin{thm}\label{t:main4}
Any geometrically formal closed oriented four-manifold admitting a (possibly non-formal) Riemannian
metric of non-negative scalar curvature is diffeomorphic to:
\begin{enumerate}
\item a rational homology sphere,
\item $\C P^2$, $S^2\times S^2$, or one of the two $S^2$-bundles over $T^2$,
\item a mapping torus $M(\varphi)$, where $\varphi\colon N\longrightarrow N$ is an orientation-preserving
self-diffeomorphism of a three-manifold admitting a metric of positive scalar curvature and $\varphi^*$ has no non-zero
fixed vector in $H^1(N;\R)$, or
\item a flat manifold.
\end{enumerate}
\end{thm}
In Subsection~\ref{ss:3psc} below we describe the three-manifolds $N$ appearing in the third case explicitly. 

Combining Theorem~\ref{t:main4} with Freedman's solution of the topological four-dimensional Poincar\'e conjecture,
we obtain:
\begin{cor}
Let $M$ be a simply connected geometrically formal closed oriented four-manifold admitting a (possibly non-formal) 
metric of non-negative scalar curvature. Then one of the following holds:
\begin{enumerate}
\item $M$ is homeomorphic to $S^4$, or
\item $M$ is diffeomorphic either to $\C P^2$ or to $S^2\times S^2$.
\end{enumerate}
\end{cor}
The standard metrics on $S^4$, on $\C P^2$, and on $S^2\times S^2$ are all formal and of constant positive Ricci
curvature. Thus the Corollary is sharp modulo the four-dimensional smooth Poincar\'e conjecture. Of course any metric on 
a fake $S^4$ would be formal, but there would not necessarily be one with non-negative scalar curvature.

Theorem~\ref{t:main4} is not optimal, since there are homology four-spheres without metrics of positive scalar curvature,
see Example~\ref{ex} in Section~\ref{s:4} below. Furthermore, it is not clear whether all the manifolds in the third case really 
have metrics of non-negative scalar curvature.
To obtain an essentially sharp result we strengthen the non-negativity of the scalar curvature to non-negativity of the Ricci curvature:

\pagebreak

\begin{thm}\label{c:Ric}
Any geometrically formal closed oriented four-manifold admitting a (possibly non-formal) Riemannian
metric of non-negative Ricci curvature is diffeomorphic to:
\begin{enumerate}
\item a rational homology sphere with finite fundamental group,
\item $\C P^2$, $S^2\times S^2$, or one of the two $S^2$-bundles over $T^2$,
\item a mapping torus $M(\varphi)$, where $\varphi$ is an orientation-preserving
isometry of a spherical space form or of $\R P^3\#\R P^3$ with their standard metrics\footnote{A standard metric on $\R P^3\#\R P^3$ is one 
that is induced by realizing $\R P^3\#\R P^3$ as an isometric quotient of $S^2\times\R$; cf.~Theorem~\ref{t:Ham} below.}, or
\item a flat manifold.
\end{enumerate}
Conversely, all these manifolds are geometrically formal. Except in the first case they all admit formal metrics with
non-negative sectional curvature.
\end{thm}
This will be proved in Section~\ref{s:4} after the proof of Theorem~\ref{t:main4}.
The improvement stems only from the assumption that $Ric\geq 0$, and has nothing to do with geometric formality.
The point is that on a manifold with $Ric\geq 0$ all harmonic one-forms are parallel by the Bochner formula, and therefore 
one obtains structure results reducing to problems in lower dimensions as soon as the first Betti number is positive. 
Since we work in dimension four, the required results in lower dimensions are known. See Corollary~\ref{c:char} for 
a characterization of non-negatively curved four-manifolds with positive first Betti number that may be of independent interest.

Theorem~\ref{c:Ric} can be compared with B\"ar~\cite[Theorem~A]{B}, where a similar result is obtained under the 
stronger assumption that $M$ carries a Riemannian metric that is {\it simultaneously} formal and of non-negative 
{\it sectional} curvature. 

In the final section we extend Theorem~\ref{t:main4} to $n$-manifolds with $b_1\geq n-2$ for all $n$.

\subsection*{Acknowledgement:} I am grateful to C.~B\"ar for stimulating correspondence.

\section{Preliminaries}\label{s:prelim}

\subsection{Geometric formality}

We recall the results about geometrically formal manifolds proved in~\cite{Formal} that we shall need here.

First of all, harmonic forms with respect to a formal metric have constant lengths and constant inner products. Therefore,
if $M$ is a geometrically formal closed oriented $n$-manifold, then its Betti numbers are bounded as follows:
$b_k(M)\leq b_k(T^n)$, and similarly $b_{2i}^{\pm}(M)\leq b_{2i}^{\pm}(T^{4i})$ if $n=4i$. Further, $b_1(M)\neq n-1$,
and if $b_1(M)=n$, then $M$ is diffeomorphic to $T^n$ and every formal metric is flat. As soon as $b_1(M)>0$,
the Euler characteristic of $M$ vanishes.

On a four-manifold every self-dual harmonic two-form with respect to a formal metric is symplectic inducing the given orientation, and every
anti-self-dual harmonic two-form is symplectic inducing the opposite orientation. This implies that a geometrically formal four-manifold with 
$b_1(M)=0$ must have $b_2^{\pm}(M)\in\{ 0, \ 1\}$; compare~\cite[Subsection~5.3]{Formal}.

\subsection{Positive scalar curvature on three-manifolds}\label{ss:3psc}

Here we specify the manifolds appearing as fibers in the mapping tori of case (3) in Theorem~\ref{t:main4}.

As explained in~\cite[Section~2]{KN}, the following is a combination of results of Schoen--Yau~\cite{SY0}, Gromov--Lawson~\cite{GL} 
and Perelman~\cite{P1,P2}, interpreted in the context of the Kneser--Milnor prime decomposition of three-manifolds:
\begin{thm}\label{t:ess}
For a closed oriented connected three-manifold $M$ the following conditions are equivalent:
\begin{enumerate}
\item $M$ is rationally inessential,
\item $M$ has no aspherical summand $M_i$ in its prime decomposition,
\item $M$ is a connected sum of copies of $S^1\times S^2$ and of spherical space forms,
\item $M$ admits a metric of positive scalar curvature.
\end{enumerate}
\end{thm}
A manifold is rationally inessential if its fundamental class maps to zero in $H_*(B\pi_1(M);\Q)$ under the classifying map of 
its universal covering.

\subsection{Strongly scalar-flat manifolds}\label{ss:ssf}

A closed manifold admitting Riemannian metrics of non-negative and not identically zero scalar curvature also has metrics 
with positive scalar curvature. 
%by the solution of the Yamabe problem. 
%In fact, for our arguments below, one often does not need this deep fact, and one could just use metrics of non-negative 
%and not identically zero scalar curvature in the same way that we use those of positive scalar curvature.
There are also manifolds which admit scalar-flat metrics, but do not admit 
any metrics of positive scalar curvature. Such manifolds are called strongly 
scalar-flat. The simplest examples are tori and other flat manifolds~\cite{GL}.

It is known that scalar-flat metrics on strongly scalar-flat manifolds 
are in fact Ricci-flat, compare for example~\cite{einstein}. In 
dimension three this implies that they are flat. 

In dimension $4$, the only known Ricci-flat manifolds are strongly scalar-flat. 
They are flat manifolds and finite quotients of $K3$ surfaces with Calabi--Yau 
metrics. The isometric quotients of $K3$ surfaces were classified by 
Hitchin~\cite{HT}, who showed that the possible covering groups are 
$\Z_{2}$ and $\Z_{2}\times\Z_{2}$, both of which do actually occur.

The following result is due to J.~Wehrheim and myself, see the Appendix to~\cite{Monopole}:
\begin{thm}\label{t:KW}
  If a closed symplectic four-manifold $M$ admits a Ricci-flat metric $g$, then $(M,g)$ 
    is isometric to a finite quotient of $T^{4}$ or $K3$ with a flat, respectively Calabi--Yau
    metric.
\end{thm}
Thus, among symplectic manifolds, there can be no other strongly scalar-flat examples.

\section{Three-manifolds}\label{s:3}

First we want to prove Theorem~\ref{t:main3}.

According to~\cite[Section~4]{Formal}, a closed orientable three-manifold $M$ is geometrically
formal if and only if it is one of the following:
\begin{enumerate}
\item any rational homology sphere,
\item a mapping torus $M=M(\varphi)$ of a surface diffeomorphism $\varphi\colon\Sigma\longrightarrow\Sigma$
with $b_1(M)=1$, or
\item the three-torus $T^3$.
\end{enumerate}
We now go through this list using Theorem~\ref{t:ess}. 
In the first case every metric is formal, and there is one with positive scalar curvature if and only if $M$ is a connected sum of 
spherical space forms.
In the second case a psc metric can only exist if the fiber $\Sigma$ is $S^2$, for otherwise $M$ would be aspherical.
But every orientation-preserving diffeomorphism $\varphi$ of $S^2$ is isotopic to the identity, and so $M$ is diffeomorphic to
to $S^1\times S^2$. The product metric is formal with positive scalar curvature. 
There is no psc metric on $T^3$.

Finally, a strongly scalar-flat three-manifold must be flat; cf.~Subsection~\ref{ss:ssf} above. Conversely, flat metrics are 
formal and trivially of non-negative scalar curvature.
This completes the proof of Theorem~\ref{t:main3}.

Note that flat manifolds appear in all three cases above, according to whether the first Betti number is $0$, $1$ or $3$.
In the second case the fiber $\Sigma$ is $T^2$ if $M$ is flat.

We can compare Theorem~\ref{t:main3} with the classification of three-manifolds of non-negative Ricci curvature
due to Hamilton~\cite{Ham}:
\begin{thm}\label{t:Ham}
A closed oriented three-manifold admits a Riemannian metric of non-negative Ricci curvature if and only if it is 
diffeomorphic to:
\begin{enumerate}
\item a spherical space form,
\item $S^1\times S^2$ or $\R P^3\#\R P^3$, or
\item a flat manifold.
\end{enumerate}
\end{thm}
Hamilton~\cite{Ham} proved that the Ricci flow deforms any metric of non-negative Ricci curvature to one of the 
model geometries $S^3$, $S^2\times\R$ or $\R^3$. Therefore, one only has to check the classification of their
closed oriented quotients; compare~\cite{S}.

Hamilton's theorem shows that in dimension three all manifolds with non-negative Ricci curvature are in fact geometrically 
formal. Indeed, their standard locally homogeneous metrics are formal.

In the case of positive first Betti number, the classifications in Theorems~\ref{t:main3} and~\ref{t:Ham}
coincide. This is no coincidence, since for non-negative Ricci curvature harmonic one-forms are parallel by Bochner's
argument, and dictate the classification without even using the Ricci flow; compare~\cite{B}. In the geometrically 
formal case harmonic one-forms may not be parallel, but they are of constant length. This, together with Gauss--Bonnet
applied to the fibers of fibrations over $S^1$, leads to the same conclusion as with parallel one-forms.

\section{Four-manifolds}\label{s:4}

In this section we prove the main results, Theorems~\ref{t:main4} and~\ref{c:Ric}. The first step is the following:
\begin{thm}\label{t:4pos}
Let $M$ be a geometrically formal closed oriented four-manifold admitting some Riemannian
metric of positive scalar curvature. Then one of the following holds:
\begin{enumerate}
\item $M$ is a rational homology sphere,
\item $M$ is diffeomorphic either to $\C P^2$, to $S^2\times S^2$, or to one of the two $S^2$-bundles over $T^2$, or
\item $M$ is a mapping torus $M(\varphi)$, where $\varphi\colon N\longrightarrow N$ is an orientation-preserving
self-diffeomorphism of a three-manifold admitting a metric of positive scalar curvature and $\varphi^*$ has no non-zero
fixed vector in $H^1(N;\R)$.
\end{enumerate}
\end{thm}
\begin{proof}
Recall that geometric formality gives strong a priori bounds on the Betti numbers, and that it implies that $M$ is symplectic
(for a suitable orientation) as soon as $b_2(M)\neq 0$. So we will start with this case.

If $b_2(M)>0$, then $M$ is a closed symplectic four-manifold admitting a psc metric. Such manifolds were classified
by Liu~\cite{L} and Ohta--Ono~\cite{OO}, who showed that they are rational or ruled. Given the Betti number bounds from
geometric formality, this means that $M$ is diffeomorphic either to $\C P^2$ or to an $S^2$-bundle over $S^2$ or $T^2$. 
For both bases there are precisely two bundles up to diffeomorphism of the total space. One of these bundles can be ruled out:
\begin{lem}\label{l:blowup}
The non-trivial $S^2$-bundle over $S^2$ is not geometrically formal.
\end{lem}
\begin{proof}
To have a convenient basis for the cohomology, we identify the non-trivial $S^2$-bundle over $S^2$ with $M=\C P^2\#\overline{\C P^2}$,
and let $H$ and $E$ be the two generators coming from the connected summands. Then $H^2=1$, $E^2=-1$, and $H\cdot E=0$.
If $M$ were geometrically formal, then every cohomology class with positive square would be represented by a symplectic form,
namely the harmonic representative with respect to a formal Riemannian metric. This applies in particular to $H$. Now the only
class with square $-1$ is $E$, with which $H$ pairs trivially. Thus, for the symplectic form in the cohomology class $H$ there 
can be no symplectic $(-1)$-sphere, meaning that the symplectic structure is minimal. This contradicts the result of 
Lalonde--McDuff~\cite{LM} to the effect that any symplectic form on $M$ is symplectomorphic to a standard K\"ahler form,
all of which are symplectically non-minimal. Thus $M$ cannot be geometrically formal.
\end{proof}
%Of the remaining manifolds, $\C P^2$, $S^2\times S^2$ and $S^2\times T^2$ are symmetric spaces that have standard metrics 
%which are both formal and of positive scalar curvature and non-negative sectional curvature. 
%The non-trivial $S^2$-bundle over $T^2$ is an isometric quotient of $S^2\times\R^2$, see~\cite[p.~167]{Ue}. 
%We will see below that its standard metric of non-negative sectional curvature is formal. 

If $b_2(M)=0$, then either $M$ is a rational homology sphere, or $b_1(M)\neq 0$. In the latter case the Euler characteristic
must vanish since harmonic one-forms with respect to formal metrics are of constant length. Thus, $b_1(M)=1$. In this 
case $M$ is a mapping torus $M(\varphi)$ for some orientation-preserving diffeomorphism $\varphi$ of a closed three-manifold
$N$, cf.~\cite[Theorem~7]{Formal}. Since $b_1(M)=1$, it follows that $\varphi^*$ has no non-zero invariant vector in $H^1(N;\R)$.
It remains to show that $N$ admits a metric of positive scalar curvature.

The total space $M$ of the mapping torus admits a psc metric by assumption. The Schoen--Yau argument~\cite{SY} therefore implies
that there is an embedded hypersurface $i\colon S\hookrightarrow M$ in the homology class of the fiber $N$ with $S$ admitting
a metric of positive scalar curvature. Let $f\colon M\longrightarrow S^1$ be the projection of the mapping torus. Then 
$(f\circ i)^*$ is trivial on $H^1(S^1;\Z)$, and so is null-homotopic. By the homotopy-lifting property of $f$, this means that $i$ 
is homotopic to a map into a single fiber $S\longrightarrow N$. This map has degree one. Since $S$ has psc, it is rationally inessential,
and so $N$ must also be rationally inessential. Thus $N$ has psc, compare Theorem~\ref{t:ess}.

This completes the proof of Theorem~\ref{t:4pos}.
\end{proof}

\begin{rem}
The attentive reader will have noticed the emphasis on {\it symplectic} minimality in the proof of Lemma~\ref{l:blowup}.
Of course the manifold $\C P^2\#\overline{\C P^2}$ carries the {\it holomorphically} minimal K\"ahler structures defined by
higher odd Hirzebruch surfaces. Although these are holomorphically minimal, they are not symplectically minimal.
The relation between the two notions of minimality was clarified in~\cite{HK}.
\end{rem}

With Theorem~\ref{t:4pos} in hand, we will complete the proof of Theorem~\ref{t:main4} by looking at strongly scalar-flat 
manifolds.
\begin{thm}\label{t:ssf}
Let $M$ be a strongly scalar-flat geometrically formal closed oriented four-manifold. Then one of the following holds:
\begin{enumerate}
\item $M$ is a rational homology sphere with finite fundamental group, or
\item $M$ is diffeomorphic to a flat manifold.
\end{enumerate}
\end{thm}
\begin{proof}
Since $M$ is strongly scalar-flat, it is Ricci-flat.

As before, if $b_2(M)>0$, then $M$ is symplectic. In this case Theorem~\ref{t:KW} tells us that $M$ is flat, or 
a quotient of a $K3$ surface by a group of order $1$, $2$ or $4$. Now a $K3$ surface and a $\Z_2$ quotient of 
it have second Betti numbers that violate the bound imposed by geometric formality. A quotient by a group of 
order $4$ has second Betti number $4<b_2(T^4)$, so this does not rule out such a quotient. However, the signature
of the quotient is $-4$, and so $b_2^-$ of this quotient is too large for it to be geometrically formal. Thus a
geometrically formal Ricci-flat manifold with positive second Betti number is flat.

If $b_2(M)=0$, then either $b_1(M)=1$, or $M$ is a rational homology sphere. If $b_1(M)=1$, then $M$ is 
an Einstein manifold with zero Euler characteristic, and so the Gauss--Bonnet formula shows that $M$ is flat. 
Finally, if $M$ is a rational homology sphere, then its Euler characteristic is positive, and so the fundamental
group is finite by the following well-known lemma, which completes the proof of the theorem.
\end{proof}

%\pagebreak

\begin{lem}\label{l:CG}
For a closed oriented four-manifold $M$ with $Ric\geq 0$ the following are equivalent:
\begin{enumerate}
\item the Euler characteristic $\chi (M)$ vanishes,
\item the first Betti number $b_1(M)$ does not vanish, and
\item the fundamental group $\pi_1 (M)$ is infinite.
\end{enumerate}
\end{lem}
\begin{proof}
This is due to Cheeger--Gromoll~\cite[Corollary~9.4]{CG1}\footnote{In~\cite{CG1} non-negative sectional curvature is 
assumed. By~\cite{CG2} the argument extends to $Ric\geq 0$.} 
and Yau~\cite[Corollary~2]{Yau}.

It is clear that each statement implies the one below it. However, the last one implies the first via the Cheeger--Gromoll
splitting theorem, cf.~\cite{CG2}, as follows. If the fundamental group is infinite, then the universal covering of 
$M$ is non-compact, and so contains a line that splits off isometrically. This implies that $M$ has a finite covering that
splits off a circle as a direct factor (diffeomorpically, not necessarily isometrically), and so has vanishing Euler characteristic.
Thus $\chi (M)=0$ by multiplicativity in coverings. 
\end{proof}

Theorems~\ref{t:4pos} and~\ref{t:ssf} together imply Theorem~\ref{t:main4} stated in the introduction.
We now show by example that not all of the manifolds listed in that theorem admit positive scalar curvature.

\begin{ex}\label{ex}
Examples of aspherical four-manifolds that are rational, respectively integral, homology spheres have been constructed 
by Luo~\cite{Luo}, respectively Ratcliffe--Tschantz~\cite{RT}. According to Schoen--Yau~\cite{SY2}, no aspherical four-manifold
allows a psc metric. The examples of~\cite{RT} are spin and non-positively curved. Therefore the existence of psc metrics on 
them is excluded by the Gromov--Lawson~\cite{GL,GL2} enlargeability obstruction. Theorem~\ref{t:ssf} implies that these 
manifolds are not strongly scalar-flat either, since they have infinite fundamental groups. Thus they do not admit any metrics 
of non-negative scalar curvature.
\end{ex}

We saw that in dimension three all manifolds with metrics of non-negative Ricci curvature are geometrically
formal. This is no longer true in dimension four. We have already seen the example of the $K3$ surface,
which is Ricci-flat but not geometrically formal. Among the del Pezzo surfaces there are also examples of 
manifolds with constant positive Ricci curvature which are not geometrically formal because their second 
Betti numbers are too large. If we consider four-manifolds with positive first Betti number, then 
there is an extension of Hamilton's Theorem~\ref{t:Ham}, inspired by the discussion of B\"ar~\cite{B}:
\begin{prop}\label{p:Ric}
Let $M$ be a closed oriented four-manifold with $b_1(M)>0$. If $M$ admits a metric of non-negative Ricci
curvature, then $M$ is diffeomorphic to:
\begin{enumerate}
\item an $S^2$-bundle over $T^2$,
\item a mapping torus $M(\varphi)$, where $\varphi$ is an orientation-preserving isometry of a spherical space form or of 
$\R P^3\#\R P^3$ with their standard metrics, or
\item a flat manifold.
\end{enumerate}
Conversely, all these manifolds admit metrics with non-negative sectional curvature.
\end{prop}
By Lemma~\ref{l:CG}, the assumption $b_1(M)>0$ is equivalent to $\pi_1(M)$ being infinite, since we are assuming $Ric\geq 0$.
\begin{proof}
The proof proceeds by the so-called Calabi construction of flat manifolds with positive first Betti number. It was observed by 
Yau~\cite{Yau} and Fischer--Wolf~\cite{FW} that this works more generally, under the assumption of non-negative Ricci curvature.

Fix a metric $g$ on $M$ with $Ric\geq 0$.
By the Bochner argument, all harmonic one-forms with respect to $g$ are parallel.
Therefore the Albanese map $\alpha_M\colon M\longrightarrow  T^{b_1(M)}$ defined by integration of harmonic forms is a 
submersion defining a smooth fiber bundle. Moreover, the metric $g$ is a local product metric for the 
orthogonal direct sum decomposition ${\mathcal V}\oplus {\mathcal H}$, where ${\mathcal H}$ are the 
tangents to the orbits of the $\R^{b_1(M)}$-action generated by the vector fields dual to the harmonic one-forms,
and ${\mathcal V}$ are the tangents to the fibers of the Albanese map. In particular the metric $g$ is flat 
on ${\mathcal H}$, and all the fibers of the Albanese carry the same induced metric. Moreover, the
induced metric on the fibers again has $Ric\geq 0$.

As in the flat, or in the geometrically formal, cases, non-negative Ricci curvature implies that $b_1(M^n)\neq n-1$.
Our four-manifold $M$ therefore has $b_1(M)\in\{ 1, \ 2, \ 4\}$. If the first Betti number is maximal, then the Albanese is 
an isometry, and $(M,g)$ is a flat $T^4$; cf.~\cite[Theorem~7]{Formal}. If $b_1(M)=2$, then the fiber of the Albanese is 
$S^2$ or $T^2$. In the latter case the induced metric on $T^2$ must be flat by Gauss--Bonnet, and so $M$ itself is flat. 

Finally, if $b_1(M)=1$, then the fiber of the Albanese is one of the three-manifolds in Hamilton's Theorem~\ref{t:Ham}. 
If the fiber is flat, then $g$ is flat. Moreover, the fiber cannot be $S^2\times S^1$ since this has the property that every 
orientation-preserving self-diffeomorphism is isotopic to the identity, and so the mapping torus would be diffeomorphic to 
$S^2\times T^2$, contradicting the assumption that $b_1(M)=1$.

We have now proved the first part of the proposition, except for the restriction on the monodromy of the mapping tori in 
case (2). We know that the fibers of the mapping tori are as claimed, and that the monodromy preserves some metric
with non-negative Ricci curvature on the fiber. By Hamilton's proof~\cite{Ham} of Theorem~\ref{t:Ham}, the Ricci flow 
deforms any such metric on a fiber to a standard metric. As the Ricci flow preserves the isometries, the monodromy will 
also be an isometry for a standard metric on the fiber.

Conversely, these mapping tori have standard metrics of non-negative sectional curvature modelled on 
$S^3\times\R$ respectively $S^2\times\R^2$ according to whether the fiber is a space form with geometry $S^3$, or 
is $\R P^3\#\R P^3$ with geometry $S^2\times\R$. For the manifolds in (1) and (3) the converse direction is 
clear, with the exception of the non-trivial $S^2$-bundle over $T^2$.

The non-trivial $S^2$-bundle over $T^2$ is a global isometric quotient of $S^2\times\R^2$ as follows; cf.~Ue~\cite[p.~167]{Ue}.
Identify $S^2$ with $\C\cup\{\infty\}$, and let the two generators of $\Z^2$ act on $S^2$ by $z\mapsto -z$ and 
$z\mapsto 1/z$, and on $\R^2$ by linearly independent translations. The resulting $S^2$-bundle over $T^2$ has a 
section of odd selfintersection, and so is non-spin. This means that it is the non-trivial bundle.
\end{proof}
This proof gives the following characterization of non-negatively curved four-manifolds with positive first Betti number in terms 
of Thurston geometries:

\pagebreak

\begin{cor}\label{c:char}
For a closed oriented four-manifold $M$ with $b_1(M)>0$ the following are equivalent:
\begin{enumerate}
\item $M$ admits a metric of non-negative sectional curvature,
\item $M$ admits a metric of non-negative Ricci curvature,
\item $M$ admits one of the Thurston geometries $S^2\times\R^2$, $S^3\times\R$ or $\R^4$.
\end{enumerate}
\end{cor}

We now have all the ingredients to prove Theorem~\ref{c:Ric}.
\begin{proof}[Proof of Theorem~\ref{c:Ric}]
Let $M$ be a geometrically formal manifold admitting some, possibly non-formal, metric with $Ric\geq 0$. Such 
a metric of course has non-negative scalar curvature, so we are in the situation of Theorem~\ref{t:main4}.
However, the assumption $Ric\geq 0$ allows (1) and (3) to be improved. 

By Lemma~\ref{l:CG}, the fundamental group is finite as soon as the first Betti number vanishes,
equivalently as soon as the Euler characteristic does not vanish. This explains the finiteness of the fundamental group
in statement (1). The improvement in case (3) follows from Proposition~\ref{p:Ric}.

For the converse direction of the theorem, we need to prove that there are indeed formal metrics on all these manifolds,
and that they can be chosen to have non-negative sectional curvature unless the manifold is a rational homology sphere.
On a rational homology sphere all metrics are formal. In case (2), $\C P^2$, $S^2\times S^2$ and $S^2\times T^2$ are 
symmetric spaces that have standard metrics which are both formal and with non-negative sectional curvature. 
By the proof of Proposition~\ref{p:Ric}, the non-trivial $S^2$-bundle over $T^2$ has a standard metric modelled on 
$S^2\times \R^{2}$. All the harmonic forms for this metric are in fact parallel, showing that the metric is formal. The same 
argument applies to flat manifolds. Finally, the mapping tori in case (3) also have the required metrics by the proof of 
Proposition~\ref{p:Ric}.
\end{proof}

\section{Higher dimensions}

The classification results we obtained in dimensions three and four have extensions to higher dimensions if we assume 
that the first Betti number is large enough compared with the dimension. 
For geometrically formal manifolds, or manifolds of non-negative Ricci
curvature, in dimension $n$ we have $b_1(M)\leq n$ with equality only if $M$ is diffeomorphic to $T^n$. In this 
case any metric that is either formal or satisfies $Ric\geq 0$ must be flat. In all other cases $b_1(M)\leq n-2$.

\begin{prop}
Let $M$ be a closed oriented $n$-manifold with $b_1(M)=n-2$. If $M$ admits a metric of non-negative Ricci
curvature, then $M$ is geometrically formal. 

Moreover, $M$ is diffeomorphic to either an $S^2$-bundle over $T^{n-2}$, or to a flat manifold. 
In the first case $M$ admits a metric of positive scalar curvature, whereas in the second case it is strongly
scalar-flat.
\end{prop}
\begin{proof}
The proof is the same as that of Proposition~\ref{p:Ric}. The fiber of the Albanese map with respect to a metric $g$
satisfying $Ric\geq 0$ is either $S^2$ or $T^2$. In the second case $g$ must be flat because the induced metric
on $T^2$ is flat by Gauss--Bonnet. In the first case $M$ admits a psc metric.
\end{proof}

By the results of~\cite{Formal}, a geometrically formal $n$-manifold with $b_1(M)=n-2$ is a surface bundle 
\begin{equation*}\label{bun}
\Sigma^2\longrightarrow M^n\longrightarrow T^{n-2}
\end{equation*}
with the property that the fundamental group of the base acts without non-zero invariant vectors on $H^1(\Sigma;\R)$.
(The second map is the Albanese of any formal metric.)
Already for $n=3$, the genus of $\Sigma$ can be any natural number. However, if we assume that $M$ admits
some, possibly non-formal, Riemannian metric of non-negative scalar curvature, then $\Sigma$ has to be a sphere 
or torus.
\begin{thm}
Let $M$ be a geometrically formal closed oriented $n$-manifold with $b_1(M)=n-2$. If $M$ admits a metric of non-negative 
scalar curvature, then:
\begin{enumerate}
\item either $M$ is flat and $\Sigma$ is a torus, or
\item $M$ is an $S^2$-bundle over $T^{n-2}$. 
\end{enumerate}
\end{thm}
\begin{proof}
If $M$ is strongly scalar-flat, then it admits a Ricci-flat metric. By the previous proposition, this metric must be flat.
It follows that $\Sigma$ is $T^2$.

It remains to show that $M$ cannot admit a metric of positive scalar curvature if the fiber $\Sigma$ of the Albanese fibration
is of positive genus. So assume $g(\Sigma)\geq 1$. Then $M$ is aspherical, and by the homotopy exact sequence 
of the Albanese fibration its fundamental group fits into an extension of the form
$$
1\longrightarrow\pi_1(\Sigma)\longrightarrow \pi_1(M)\longrightarrow \Z^{n-2}\longrightarrow 1 \ .
$$
Since the strong Novikov conjecture holds for $\pi_1(\Sigma)$, it also holds for $\pi_1(M)$, see Rosenberg~\cite[Proposition~2.5]{Ros}.
It follows that the aspherical manifold $M$ cannot have positive scalar curvature, essentially by the results of 
Gromov--Lawson~\cite{GL2}, compare also Rosenberg~\cite[Theorem~3.5]{Ros}.
\end{proof}

The results of this section should have extensions to $n$-manifolds with $b_1(M)=n-3$ since in dimension three we have a complete
understanding of manifolds with $Ric\geq 0$, and even with non-negative scalar curvature, compare Sections~\ref{s:prelim}
and~\ref{s:3} above. No such extension is possible for $b_1(M)=n-4$, because there are just too many four-manifolds with positive 
scalar curvature, and maybe even with positive Ricci curvature.

\bigskip

\bibliographystyle{amsplain}

\end{document}